\documentclass[12pt]{amsart}

\usepackage{amsthm, mathrsfs,amssymb,amsmath,}
\usepackage{enumerate}

\makeatletter
\@namedef{subjclassname@2010}{%
\textup{2010} Mathematics Subject Classification}
\makeatother

\allowdisplaybreaks

\newtheorem{thm}{Theorem}[section]
\newtheorem{lem}[thm]{Lemma}
\newtheorem{prop}[thm]{Proposition}
\newtheorem{cor}[thm]{Corollary}

\theoremstyle{definition}
\newtheorem{defn}[thm]{Definition} 

\newtheorem{lemma}{Lemma}[section]
\theoremstyle{remark}

\def\C{{\mathbb C}}

\def\Z{{\mathbb Z}}

\def\Zn{{\mathbb Z}^n}

\def\Rn{{\mathbb R}^n}
\def\R2n{{\mathbb R}^{2n}}

\def\S{{\mathcal S}}

\def\Rn{{\mathbb R}^n}

\def\C{{\mathbb C}}

\def\Z{{\mathbb Z}}

\def\R^2{{\mathbb R}^2}
\def\R2n{{\mathbb R}^{2n}}
\def\R{{\mathbb R}}

\def\C{{\mathbb C}}
\def\S{{\mathcal S}}

\def\N{{\mathbb N}}

\title[Ellipticity and Fredholmness of pseudo-differential operators on 
$\ell^2(\Zn)$]{Ellipticity and Fredholmness
of pseudo-differential operators on $\ell^2(\Zn)$ }
\author{Aparajita Dasgupta}
\address{Aparajita Dasgupta \endgraf Department of Mathematics \endgraf Indian Institute of Technology Delhi \endgraf New Delhi - 110 016, India.}
\email{adasgupta@maths.iitd.ac.in}

\author{Vishvesh Kumar} 

\address{Vishvesh Kumar  \endgraf Department of Mathematics: Analysis, Logic and Discrete Mathematics
	\endgraf Ghent University
	\endgraf Krijgslaan 281, Building S8,	B 9000 Ghent,
	Belgium .} 
\email{vishveshmishra@gmail.com}

\begin{document}
\begin{abstract} 
The minimal operator and the maximal operator of an elliptic 
pseudo-differential operator with symbols on $\Z^n\times \mathbb{T}^n$ are 
proved to coincide and the domain is 
given in terms of a Sobolev space. Ellipticity and Fredholmness are proved 
to be equivalent for pseudo-differential operators on $\Z^n$. The index of 
an elliptic pseudo-differential operator on $\Z^n$ is also computed.

 \end{abstract}


\maketitle


\section{Introduction}
\setcounter{equation}{0}
The aim of this paper is to investigate the ellipticity and Fredholmness of 
pseudo-differential 
operators on $\ell^2({\Zn})$ in the context of minimal and maximal operators, 
the domains of elliptic pseudo-differential operators and Fredholm operators. 
The pseudo-differential operators on the lattice $\mathbb{Z}^n$, as in this 
paper, are suitable for solving difference equations 
on $\mathbb{Z}^n.$ Such equations naturally appear in various problems 
of modelling and in the discretisation of continuous problems.  Several 
attempts of developing a suitable theory of pseudo-differential operators 
on the lattice $\mathbb{Z}^n $ have been done in the literature, see e.g. 
\cite{DW13, GJBNM16, SW, Rab10, RR09} but with no symbolic calculus.  
Recently,  a global symbolic calculus has been developed in \cite{Rulat}. 
The symbol 
classes therein exhibit improvements when differences are taken with 
respect to 
the space (lattice) variable, thus resembling the behaviour of the 
so-called SG 
pseudo-differential operators on $\mathbb{R}^n$, first developed by Cordes 
\cite{Cor95}, but with a twist in variables. Ellipticity is a condition on the operators at infinity where the lower-order terms cannot be neglected. On the other hand, Fredholmness measures the almost invertibility of the operators and depends very much on the space on which
the operators are defined. For a good theory of pseudo-differential operators on a given function space, a relation between of ellipticity and Fredholmness property is
crucial.  In \cite{AW2}, the ellipticity of the Fredholm pseudo-differential
operator with arbitrary SG-symbols  has been proved.
For some recent developments 
on the theory of pseudo-differential operators on $\Zn$ and its 
applications, we refer 
to \cite{CK, Duvan}. In  \cite{ADWo, Wongbook}, the equivalence of ellipticity and Fredholmness and the applications of SG 
pseudo-differential operators with positive order on $L^{p}(\mathbb{R}^n)$, 
$1<p<\infty,$ has been investigated.  The maximal and minimal extension for another class of elliptic pseudo-differential operators on $L^{p}(\mathbb{R}^n),$  
$1<p<\infty$, can be found in \cite{Wo}. 
In Section \ref{Pre}, we recall symbol classes and the corresponding calculus 
of the pseudo-difference operators on $\mathbb{Z}^n.$ An interesting 
difference with the usual theory of pseudo-differential operators is that 
since the space $\Zn$ is discrete, the Schwartz kernels of the corresponding pseudo-differencial  operators do not have singularities at the diagonal.  Also the phase space is $\Zn\times \mathbb{T}^n$ with the frequencies being elements of the compact space $\mathbb{T}^n$ (the torus $\mathbb{T}^n:=\Rn
/\Zn)$. The usual theory of pseudo-differential operators works with symbol classes with increasing decay of symbols after taking their derivatives in the frequency variable. The discrete calculus is essentially different since one can not construct it using standard methods relying on the decay properties in the frequency component of the phase space since the frequency space, $\mathbb{T}^n$ is compact, so no improvement with respect to the decay of the frequency variable is possible.
Also it is not suitable to work with derivatives with respect to the space variable $k\in\Zn$. Therefore, it is replaced by working with appropriate difference operators on the lattice.  The global theory of discrete symbolic calculus has been developed in \cite{Rulat} similar to the theory developed in \cite{R2, RT, R7, R6,  R1, R3, R4} but with a twist, swapping the order of the space and frequency variables.  In Section \ref{Pre}, we developed the family of discrete Sobolev  spaces, which are an essential tools to study the maximal and minimal extensions of the discrete operators.  Furthermore, in Section \ref{maxmin}, we investigate the  minimal and maximal  operators of pseudo-differential operators on $\mathbb{Z}^n$ with symbols in $S^{m},$ $m>0$ and show that they are equal  for elliptic operators. The main ingredient  for this is an analogue of the Agmon--Douglis--Nirenberg inequalities for elliptic discrete 
pseudo-differential operators, which we also establish in Section 
\ref{maxmin}. In Section \ref{Fredell}, we first recall the Fredholm pseudo-differential operators and study Fredholmness property of discrete pseudo-differential operators and calculated the index of a Fredholm 
pseudo-differential operator on $\ell^2({\mathbb{Z}}^n)$.  

\section{Preliminaries}\label{Pre}
In this section, we recapitulate basic facts, notation and definitions of discrete Fourier analysis and pseudo-differential operator on $\mathbb{Z}^n.$ We begin with recalling the definition of Fourier transform on lattice $\mathbb{Z}^n.$

\setcounter{equation}{0}
The discrete Fourier transform $\hat{f}$ of a function $f$ in 
$\ell^1(\mathbb{Z}^n)$ is defined by 
$$\widehat{f}(x)= \sum_{k \in \mathbb{Z}^n} 
e^{-2 \pi i k \cdot x} \,f(k)\, $$ for all $x \in \mathbb{T}^n$, where
${\mathbb{T}}^n=\R^n/{\Z^n},$ $k \cdot x= \sum_{j=1}^n k_jx_j$ 
for $k=(k_1, k_2, \ldots, k_n) \in \Z^n$ and $x=(x_1, x_2, \ldots, x_n) 
\in \mathbb{T}^n.$ The discrete Fourier transform can be extended to 
$\ell^2(\mathbb{Z}^n)$ using the usual  density arguments. We normalize the 
Haar 
measures on $\Z^n$ and $\mathbb{T}^n$ in such  a way that the  
Plancherel formula to the effect that
$$ \sum_{k \in \Z^n} |f(k)|^2 = \int_{\mathbb{T}^n} |\widehat{f}(x)|^2\,dx.
$$ is valid. Next, the inverse discrete Fourier transform is given by 
$$f(k)= \int_{\mathbb{T}^n} e^{2 \pi i k \cdot x} \widehat{f}(x) \, 
dx,\quad k\in \Z^n.$$ 

We recall the discrete calculus  developed in \cite{RT}.
 Let $f$ be a function on $ \Z^n$ and $e_j \in \N^n$ be such that $e_j$ 
has $1$ in the $j^{th}$ entry and zeros elsewhere.  The l 
difference operator $\Delta_{k_j}$ is defined by 
 $$\Delta_{k_j} f(k)= f(k+e_j)- f(k)$$ and set 
$$\Delta_{k}^\alpha= \Delta_{k_1}^{\alpha_1} \Delta_{k_2}^{\alpha_2} 
\ldots \Delta_{k_n}^{\alpha_n}$$ for all 
$$\alpha=(\alpha_1,\alpha_2,\dots, 
\alpha_n) \in \N^n_0=\N^n\cup\{0\}$$ and all
$$k=(k_1,k_2,\dots,k_n)\in \Z^n.$$ 
We have the following formulae from \cite{Rulat}, that  are useful in the sequel. 
 \begin{itemize}
 	\item[(i)] Let $f: \Z^n \rightarrow \C.$ Then
 	$$(\Delta_{k}^\alpha \sigma)(k)= \sum_{\beta \leq \alpha} 
(-1)^{|\alpha-\beta|} \binom{\alpha}{\beta} f(k+\beta),\quad k\in \Z^n.$$
 	 \item[(ii)]  Let $f, g : \Z^n \rightarrow \C.$ Then 
 	$$(\Delta_{k}^\alpha(fg))(k)= \sum_{\beta \leq \alpha} 
\binom{\alpha}{\beta} \, (\Delta_{k}^\beta f)(k)\, 
(\bar{\Delta}_{k}^{\alpha-\beta} g)(k+\beta),\quad k\in \Z^n,$$ for all 
multi-indices $\alpha.$ This is a special case of the Leibniz formula.  
 	 \item [(iii)] Let $f, g : \Z^n \rightarrow \C.$ Then 
 	 $$ \sum_{k \in \Z^n} f(k)\, (\Delta_{k}^\alpha g)(k)= 
(-1)^{|\alpha|} \sum_{k \in \Z^n} (\bar{\Delta}_k^\alpha f)(k)\, 
g(k),\quad k\in \Z^n,$$ 
where $(\bar{\Delta}_{k_j}f)(k)= f(k)-f(k-e_j),$ with higher order 
differences iteratively defined. 
 \end{itemize} 
 We will be using the usual notations,
 $$D^{\alpha}_{x}=D^{\alpha_1}_{x_1}...D^{\alpha_n}_{x_n}, ~~~D_{x_j}=\frac{1}{2\pi i}\frac{\partial}{\partial x_j}$$ and 
 $$D^{(\alpha)}_{x}=D^{(\alpha_1)}_{x_1}...D^{(\alpha_n)}_{x_n}, ~~~D_{x_j^{l}}=\prod\limits_{m=0}^{l}\left(\frac{1}{2\pi i}\frac{
 \partial}{\partial x_j}-m\right),~~~l\in\N.$$
The operators $D^{(\alpha)}_{x}$ are useful in the analysis in torus and details can be found in \cite{RT}. The symbol classes  are then defined as follows: 
 \begin{defn} For $-\infty <m <\infty,$ we say that a function 
$\sigma: \Z^n \times \mathbb{T}^n \rightarrow \C$ belongs to 
$S^m(\Z^n \times \mathbb{T}^n)$ if $\sigma (k, \cdot) \in 
C^\infty(\mathbb{T}^n)$ for all $k \in \Z^n$ and for all multi-indices 
$\alpha, \beta$ there exists a positive constant $C_{\alpha, \beta}$ such that 
	$$|(D_x^{(\beta)} \Delta_{k}^\alpha \sigma)(k, x)| \leq C_{\alpha, 
\beta} (1+|k|)^{m-|\alpha|},\quad (k,x)\in \Z^n\times \mathbb{T}^n.$$
\end{defn}

The corresponding discrete pseudo-differential operator with symbol $\sigma$ is given by
$$(T_\sigma f)(k)= \int_{\mathbb{T}^n} e^{2 \pi i k \cdot x}\, 
\sigma(k, x)\widehat{f}(x)\, dx, \quad k \in \Z^n.$$

The following theorem gives  the product of two discrete 
pseudo-differential operators. 
\begin{thm} \label{compo} \cite[Theorem 3.1]{Rulat} Let $\sigma \in 
S^{m_1}(\Z^n \times \mathbb{T}^n)$ and $\tau \in S^{m_2}
(\Z^n \times \mathbb{T}^n).$ Then the product $T_\sigma  T_\tau$ of the 
pseudo-differential operators $T_\sigma$ and $T_\tau$ is a 
pseudo-differential operator with symbol in $S^{m_1+m_2}(\Z^n \times \mathbb{T}^n).$
\end{thm}

\begin{defn}
	A symbol $\sigma \in S^m(\Z^n \times \mathbb{T}^n)$ is called {\it 
elliptic} (of order $m$) if there exist positive constants $C$ and  $M>0$ 
such that 
	$$|\sigma(k, x)| \geq C(1+|k|)^m,$$ for all $k\in \Z^n$ and all 
$x\in \mathbb{T}^n$ 
with $|k|>M$.  
\end{defn}
The corresponding 
pseudo-differential operator $T_\sigma$ is called {\it elliptic}.\\
The following theorem gives the parametrix for an elliptic 
pseudo-differential operators. 
\begin{thm} \label{Ruz3.6} \cite[Theorem 3.6]{Rulat} Let $\sigma \in 
S^m(\Z^n \times \mathbb{T}^n),\, -\infty <m<\infty,$ be elliptic. Then 
there exists a symbol $\tau \in S^m(\Z^n \times \mathbb{T}^n)$ such that 
$$T_\sigma T_\tau= I+R$$ $$ T_\tau T_\sigma = I+S$$ where $R$ and $S$ are 
infinitely smoothing in the sense that they are pseudo-differential operators 
with symbol in $\bigcap_{\nu \in \R} S^\nu (\Z^n \times \mathbb{T}^n).$
	\end{thm}

Next we recall the {\it Schwartz space} $\S(\Z^n)$, on the lattice $\Zn$,  the space of all functions 
$\varphi: 
\Z^n \rightarrow \C$ such that for all multi-indices $\alpha$ and $\beta$,
$$\sup_{k\in \mathbb{Z}^n } |k^\alpha (\Delta_{k}^\beta 
\varphi)(k)|<\infty.$$
A sequence $\{\varphi_j\}$ of functions in $\S(\Z^n)$ is said to be 
converge 
to $0$ in $\S(\Z^n)$ if for all multi-indices $\alpha, \beta$, we have 
$$ \sup_{k\in \mathbb{Z}^n } |k^\alpha(\Delta_{k}^\beta \varphi_j)(k)| 
\to 0$$ as $j\to \infty.$
A linear functional $T$ on $\S(\Z^n)$ is called a tempered distribution if 
for any sequence $\{\varphi_j\}$ of function in $\S(\Z^n)$ converging to 
$0,$ we have $T(\varphi_j) \to 0$ as $j \to \infty.$ 

For $-\infty <s< \infty,$ we denote by $J_s$ the pseudo-differential operator 
of which the symbol $\sigma_s$ is given by 
$$\sigma_s(k)= (1+|k|^2)^{s/2},\quad k \in \Z^n.$$
Note that the symbol of $J_s$ is in $S^{s}(\Z^n \times \mathbb{T}^n).$ 
The 
pseudo-differential operator $J_s$ is often called the discrete Bessel 
potential of order $s.$

Now for $-\infty <s< \infty,$ we define $l^2$-Sobolev space, $H^{s, 2}(\Z^n),$ 
to be the set of all tempered distributions $u$ for which $J_{-s}u$ is 
in $\ell^2(\Z^n).$ Then $H^{s, 2}(\Z^n)$ is a Banach space with respect to the 
norm $\|\cdot\|_{s, 2}$ given by 
$$\|u\|_{s,2} = \|J_{-s}u\|_2, \quad u \in H^{s,2}(\Z^n).$$ It is 
easy to see that for $-\infty<s,\,t<\infty,$ $J_t$ is an isometry from 
$H^{s,2}(\Z^n)$ onto $H^{s+t, 2}(\Z^n).$ 

\section{Minimal and Maximal Pseudo-differential operators on $\mathbb{Z}^n$}\label{maxmin}
This section is devoted to investigation of the  minimal and maximal  operators of pseudo-differential operators on $\mathbb{Z}^n$ with symbols in $S^{m},$ $m>0.$ and show that they are equal  for elliptic operators. For this purpose, we prove for  an analogue of the Agmon--Douglis--Nirenberg inequalities for elliptic discrete 
pseudo-differential operators.

\setcounter{equation}{0}
The following theorem gives the relation between lattice and toroidal 
quantizations proved in \cite[Theorem 4.1]{Rulat}.
\begin{thm} \label{ruz}
Let $\sigma: \mathbb{Z}^n \times \mathbb{T}^n 
\rightarrow \mathbb{C}$ be a measurable function such that the pseudo-differential 
operator $T_\sigma:l^{2}(\Z^n)\to \ell^2(\Zn)$ is a bounded linear operator. 
If we define $\tau: \mathbb{T}^n \times \mathbb{Z}^n \to
\mathbb{C}$ by $$\tau(x, k)= \overline{\sigma(-k, x)}, \quad k\in \Z^n, 
x\in \mathbb{T}^n,$$ then $T_\sigma= 
\mathcal{F}_{\mathbb{Z}^n} T_\tau^* \mathcal{F}_{\mathbb{Z}^n}^{-1},$ where $T_\tau^*$ is the adjoint of $T_\tau.$  
\end{thm}
\begin{thm} \label{l^2}
	Let $\sigma \in S^0(\mathbb{Z}^n \times \mathbb{T}^n).$ Then 
$T_\sigma: \ell^2(\mathbb{Z}^n) \to \ell^2(\mathbb{Z}^n)$ is a bounded 
linear operator.  
\end{thm} 
\begin{proof} Using Theorem \ref{ruz} and the fact that the Fourier transform 
$\mathcal{F}_{\mathbb{Z}^n}:\ell^2(\mathbb{Z}^n)\to L^2(\mathbb{T}^n)$ is an 
isometry,  
it follows that $T_\sigma:\ell^2(\mathbb{Z}^n)\to \ell^2(\Z^n)$ is a bounded 
linear operator if and only if $T_\tau: L^2(\mathbb{T}^n)\to 
L^2(\mathbb{T}^n)$ is a bounded linear operator.  Now note that $\sigma 
\in S^0(\mathbb{Z}^n \times 
\mathbb{T}^n)$ if and only if $\tau \in S^0(\mathbb{T}^n \times 
\mathbb{Z}^n).$ But by \cite[Proposition 2.6]{SW} 
$T_\tau:L^2(\mathbb{T}^n)\to L^2(\mathbb{T}^n)$ is a bounded linear 
operator if 
$\tau \in S^0(\mathbb{T}^n \times \mathbb{Z}^n).$ So, if  
$\sigma \in S^0(\mathbb{Z}^n \times \mathbb{T}^n)$,  then 
$T_\sigma:\ell^2(\Z^n)\to \ell^2(\Z^n)$ is a bounded linear operator. 
  \end{proof}

\begin{thm} \label{vis4.3}
	Let $\sigma \in S^m(\mathbb{Z}^n \times \mathbb{T}^n),$ 
$-\infty < m < \infty.$ Then $T_\sigma: H^{s,2}(\mathbb{Z}^n) \to 
H^{s-m,2}(\mathbb{Z}^n)$ is a bounded linear operator. 
\end{thm}
\begin{proof}
	Let $u \in H^{s, 2}.$ Then $J_{-s}u \in \ell^2(\mathbb{Z}^n)$ and $$ 
\|T_\sigma u\|_{s-m, 2} = \|J_{-s+m}(T_\sigma u)\|_{\ell^2(\mathbb{Z}^n)} = 
\|(J_{-s+m} T_\sigma J_s)(J_{-s}u)\|_{\ell^2(\mathbb{Z}^n)}.$$
	Since $J_{-s+m}T_\sigma J_s$ is a pseudo-differential operator 
with symbol in $S^0$ by Theorem \ref{compo} and therefore using 
Theorem \ref{l^2}, there exists a positive constant $C$ such that 
	$$\|T_\sigma u\|_{s-m, 2} \leq C \|J_{-s}u\|_{\ell^2(\mathbb{Z}^n)}= 
C \|u\|_{s, 2}.$$ Hence $T_\sigma: H^{s,2}(\mathbb{Z}^n) 
\to H^{s-m,2}(\mathbb{Z}^n)$ is a bounded linear operator. 
\end{proof}

\begin{thm} \label{vis4.4}
	Let $s \leq t.$ Then $H^{t,2}(\mathbb{Z}^n) 
\subseteq  H^{s,2}(\mathbb{Z}^n).$ In fact, there exists a positive 
constant $C$ such that $$\|u\|_{s,2} \leq C \|u\|_{t,2},\quad u\in 
H^{t,2}(\Z^n).$$ 
\end{thm}
\begin{proof}
	Let $u \in H^{t, 2}(\mathbb{Z}^n).$ Then we have 
	$$\|u\|_{s, 2} = \|J_{-s} u\|_{\ell^2(\mathbb{Z}^n)} = 
\|J_{t-s}J_{-t}u\|_{\ell^2(\mathbb{Z}^n)}.$$ Since $s \leq t,$ it 
follows that $J_{t-s}$ is a 
pseudo-differential operator with symbol in $S^0.$ So, by Theorem \ref{l^2}, 
there exists a positive constant $C$ such that 
	$$\|J_{-s}u\|_{\ell^2(\mathbb{Z}^n)} = 
\|J_{t-s}J_{-t}u\|_{\ell^2(\mathbb{Z}^n)} \leq C 
\|J_{-t}u\|_{\ell^2(\mathbb{Z}^2)} = C \|u\|_{t, 2}, \quad  u \in 
H^{t,2}(\mathbb{Z}^n).$$
	Thus, 
	    $$\|u\|_{s,2} \leq C \|u\|_{t, 2},\quad  u \in 
H^{t, 2}(\mathbb{Z}^n).$$
\end{proof} 
\begin{thm}\label{compact}
Let 
$s,t\in (-\infty,\infty)$ be such that $s< t.$ Then the inclusion $i: 
H^{t,2}(\Zn)\to H^{s,2}(\Zn)$ is a compact operator.
\end{thm}
To prove the above theorem, we need to recall pseudo-diffrential operators 
with 
symbols introduced by Grushin \cite{Gru}. For $m\in (-\infty,\infty),$ 
let 
$S^{m}_{0}$ be the set of all functions $
\sigma$ in $C^{\infty}(\Rn\times\Rn)$ such that for all multi-indices 
$\alpha$ and $\beta$, there exists a bounded real-valued  function 
$C_{\alpha,\beta}$ on $\Rn$ for which 
\begin{equation}\label{con1}
|(D^{\alpha}_{x}D^{\beta}_{\xi}\sigma)(x,\xi)|\leq C_{\alpha,\beta}(x)
(1+|\xi|)^{m-|\beta|},\quad x,\xi\in \Rn,
\end{equation}
and 
\begin{equation}\label{con2}
\lim_{|x|\to\infty} C_{\alpha,\beta}(x)=0
\end{equation}
for $|\alpha|\neq 0.$
For $-\infty<m<\infty,$ we can now define $S_0^m(\Z^n\times \mathbb{T}^n)$ 
to be the set of of symbols $\sigma$ in $S^m(\Z^n\times\mathbb{T}^n)$ such 
that for 
all multi-indices $\alpha$ and $\beta$ there exists a bounded real-valued 
function $C_{\alpha,\beta}$ on $\Z^n$ for which
$$|(D_x^{(\beta)}\Delta_k^\alpha\sigma)(k,x)|\leq 
C_{\alpha,\beta}(1+|k|)^{m-|\alpha|},\quad (k,x)\in \Z^n\times 
\mathbb{T}^n,$$ 
and
$$\lim_{|k|\to \infty}C_{\alpha,\beta}(k)=\lim_{|k|\to \infty}C_{\alpha,\beta}(1+|k|)^{m-|\alpha|}=0.$$ In particular $S^{-m}(\Zn\times\mathbb{T}^n)\subseteq S_0^m(\Z^n\times \mathbb{T}^n) .$
Then the following  discrete $\ell^2$-version of  \cite[Theorem 3.2]{Wo18} 
for discrete pseudo-differential operators can be proved using similar techniques.
\begin{thm} \label{vis3.6}
Let $\sigma\in S_0^{0}(\mathbb{Z}^n\times\mathbb{T}^n).$  Then for every 
positive number $\varepsilon,$ $T_{\sigma}: H^{s+m,2}(\Zn)\to 
H^{s-\varepsilon,2}(\Zn)$ is a compact operator for $-\infty<s<\infty.$
\end{thm}

As a simple consequence of Theorem \ref{vis3.6}, we have the following 
corollary. 
\begin{cor}\label{Besco}
For every positive number $\varepsilon$, $J_{\varepsilon}: 
\ell^2(\Zn)\rightarrow \ell^2(\Zn)$ is a compact operator.
\end{cor}
The proof of Theorem \ref{compact} again follows from the above 
Corollary \ref{Besco} using the same techniques as in Theorem 2.5 of 
\cite{ADWo}.
\begin{proof}{\bf  of Theorem \ref{compact}}
Let $\epsilon>0$, be such that
$$t-s-\epsilon>0.$$ Since $J^{-1}_{\epsilon}J_{-s}$ is a discrete pseudo-differential operator of order $s+\epsilon$ , it follows that the composition $J_{\epsilon}iJ^{-1}_{\epsilon}J_{-s}$ of the mappings
$$J^{-1}_{\epsilon}J_{-s}: H^{t,2}\rightarrow H^{t-s-\epsilon,2}$$
$$i:H^{t-s-\epsilon,2}\hookrightarrow \ell^{2}(\mathbb{Z}^n),$$
and $$J_{\epsilon}:\ell^{2}(\mathbb{Z}	^n)\rightarrow \ell^{2}(\mathbb{Z}^n)$$ is compact since $J^{-1}_{\epsilon}J_{-s}: H^{t,2}\rightarrow H^{t-s-\epsilon,2}$ is a bounded linear operator by Theorem \ref{vis3.6} and $J_{\epsilon}:\ell^{2}(\mathbb{Z}	^n)\rightarrow \ell^{2}(\mathbb{Z}^n)$ is a compact operator by Corollary \ref{Besco}. Thus the linear operator 
$$H^{t,2}\ni u\rightarrow J_{\epsilon}iJ_{\epsilon}^{-1}J_{-s}u=J_{-s}\in \ell^{2}(\Zn)$$ is compact and this completes the proof.
\end{proof}
\begin{prop} \label{4.5}
	Let $\sigma \in S^m (\mathbb{Z}^n \times \mathbb{T}^n).$ Then the 
pseudo-differential operator $T_\sigma:\S(\Z^n)\to \ell^2(\Zn)$ is closable. 
\end{prop}
\begin{proof}
	Let $\{\varphi_j\}$ be a sequence  of functions in $\mathcal{S}
(\mathbb{Z}^n)$ such that $\varphi_j \to 0$ and $T_\sigma 
\varphi_j \to f$ in $\ell^2(\Zn)$ as $j \rightarrow 
\infty.$ Then for any function $\psi \in \mathcal{S}(\mathbb{Z}^n)$, we have 
	$$( T_\sigma \varphi_j, \psi)_{\ell^2(\Zn)} = ( \varphi_j, 
T_\sigma^* \psi)_{\ell^2(\Zn)},\quad j=1,2,\ldots,$$ where $T_\sigma^*$ 
is the formal  
adjoint of $T_\sigma$. Let $j \to \infty.$ Then $( f, \psi)_{\ell^2(\Zn)} 
=0$ for all 
functions $\psi \in \mathcal{S}(\mathbb{Z}^n).$ Since 
$\mathcal{S}(\mathbb{Z}^n)$ is dense in $\ell^2(\Zn),$ it follows 
that $f=0.$ Hence $T_\sigma:\S(\Z^n)\to \ell^2(\Zn)$ is closable.    
\end{proof}
 
 As a consequence of Proposition \ref{4.5}, the minimal operator 
$T_{\sigma, 0}$ of $T_\sigma$ exists. Let us recall that the domain 
$\mathcal{D}(T_{\sigma, 0})$ of $T_{\sigma, 0}$ consist of all functions $u$ 
in $\ell^2(\mathbb{Z}^n)$ for which a sequence $\{\varphi_j\}$ in 
$\mathcal{S}(\mathbb{Z}^n)$ can be found such that $\varphi_j \to 
u$ in $\ell^2(\Zn)$ and $T_\sigma \varphi_j \to f$ in 
$\ell^2(\Zn)$ for some $f \in \ell^2(\Zn)$ as $j 
\to \infty.$ Also, we have $T_{\sigma, 0} u=f.$
 
 \begin{defn} \label{def4.6}
 	Let $u$ and $f$ be two functions in $\ell^2(\Zn).$ We say that 
$u$ lies in $\mathcal{D}(T_{\sigma, 1})$ and $T_{\sigma,1}u=f$ if and only if
 	\begin{equation} \label{eq1new}
 	( u, T_{\sigma}^*\varphi)_{\ell^2(\Zn)}= ( f, 
\varphi)_{\ell^2(\Zn)}, \quad \varphi \in \mathcal{S}(\mathbb{Z}^n),
 	\end{equation}
 	where $T_\sigma^*$ is the formal adjoint of $T_\sigma.$
 \end{defn}
\begin{prop} \label{vis4.7}
	Let $u \in \mathcal{D}(T_{\sigma, 1}).$ Then $T_{\sigma, 1}u 
= T_\sigma u$ in the distribution sense.
\end{prop}
\begin{proof} By Definition \ref{def4.6}, we have 
$$ (u, T_{\sigma}^*\varphi)_{\ell^2(\Z^n)}= (T_{\sigma,1}u, 
\varphi)_{\ell^2(\Z^n)},
\quad  \varphi \in \mathcal{S}(\mathbb{Z}^n).$$ So, by considering $u$ and 
$T_{\sigma, 1}u$ as tempered distributions, we get 
\begin{equation} \label{eq1}
(T_{\sigma,1}u)(\bar{\varphi})= u(\overline{T_\sigma^* \varphi}),\quad 
\varphi \in \mathcal{S}(\mathbb{Z}^n). 
\end{equation}
On the other hand, we have \begin{equation} \label{eq2}
(T_\sigma u)(\bar{\varphi})= u(\overline{T_\sigma^* \varphi}),\quad 
\varphi \in \mathcal{S}(\mathbb{Z}^n).
\end{equation}
            Hence, by \eqref{eq1} and \eqref{eq2}, $T_{\sigma, 1}u= T_\sigma u$ in the distribution sense.
	\end{proof}
\begin{thm}
	$T_{\sigma,1 }$ is a closed linear operator from 
$\ell^2(\Zn)$ into $\ell^2(\Zn)$ with domain 
$\mathcal{D}(T_{\sigma, 1})$ containing $\mathcal{S}(\mathbb{Z}^n).$
\end{thm}
\begin{proof}
	It is clear from the definition of the formal adjoint $T_\sigma^*$ 
and \eqref{eq1new} that $\mathcal{S}(\mathbb{Z}^n) \subset  
\mathcal{D}(T_{\sigma, 1}).$ It is easy to prove the linearity. Now we prove 
that $T_{\sigma, 1}$ is closed. Let $\{u_j\}$ be a sequence of 
functions in $\mathcal{D}(T_{\sigma,1})$ such that $u_j \to u$ in 
$\ell^2(\Zn)$ and $T_{\sigma, 1}u_j \to f$ in 
$\ell^2(\Zn)$ for some $u$ and $f$ in $\ell^2(\Zn)$ as $j 
\rightarrow \infty.$ Then by \eqref{eq1new}
	\begin{equation} \label{eq4}
	(u_j, T_\sigma^* \varphi)_{\ell^2(\Z^n)} = ( T_{\sigma, 1} u_j, 
\varphi )_{\ell^2(\Z^n)}
	\end{equation} for all $\varphi \in \mathcal{S}(\mathbb{Z}^n)$ and 
$j=1,2,\ldots.$  By letting $j \to \infty$ in \eqref{eq4}, we have 
	$$(u, T_\sigma^* \varphi)_{\ell^2(\Z^n)} = (f, \varphi)_{\ell^2(\Z^n)}, 
\quad  \varphi \in \mathcal{S}(\mathbb{Z}^n). $$ Hence by Definition 
\ref{def4.6}, $u \in \mathcal{D}(T_{\sigma, 1})$ and $T_{\sigma, 1} u=f.$ This proves that $T_{\sigma, 1}$ is closed. 
\end{proof}

\begin{prop} \label{visext}
	$T_{\sigma, 1}$ is an extension of $T_{\sigma, 0}.$
\end{prop}
\begin{proof}
	Let $u \in \mathcal{D}(T_{\sigma, 0})$ and $T_{\sigma, 0}u=f.$ 
Then there exists a sequence $\{\varphi_j\}$ of functions in $\mathcal{S}$ 
for which $\varphi_j \to u$ and $T_\sigma \varphi_j \to f$ 
in $\ell^2(\Zn)$ as $j \to \infty.$ By the definition of 
$T_\sigma^*$, we have 
	$$( \varphi_j, T_\sigma^*)_{\ell^2(\Zn)}= ( T_\sigma 
\varphi_j, \psi)_{\ell^2(\Zn)}$$ for all $\psi \in 
\mathcal{S}(\mathbb{Z}^n)$ and 
$j = 1,2,\ldots.$ By letting $j\to \infty,$ we get $$( u, 
T_\sigma^*\psi)_{\ell^2(\Zn)}= ( 
f, \psi)_{\ell^2(\Zn)},\quad  \psi \in \mathcal{S}(\mathbb{Z}^n).$$ So, 
by Definition 
\ref{def4.6}, $u \in \mathcal{D}(T_{\sigma, 1})$ and $T_{\sigma, 1}u=f.$ 
\end{proof}

\begin{lem} \label{lem10}
	$T_\sigma^t \varphi = T_\sigma^* \varphi$ for all $\varphi 
\in \mathcal{S}(\mathbb{Z}^n),$ where $T_\sigma^t$ is the true adjoint of 
$T_\sigma.$  In other words, the true adjoint 
and the formal adjoint coincide on the space $\mathcal{S}(\mathbb{Z}^n).$ 
\end{lem}
\begin{proof}
	Let $\varphi \in \mathcal{S}(\mathbb{Z}^n).$ Then by the definition 
of $T_\sigma^*,$ we have $$(T_\sigma^* \varphi, \psi)_{\ell^2(\Zn)}= ( 
\varphi, T_\sigma \psi)_{\ell^2(\Zn)}, \quad  \psi \in 
\mathcal{S}(\mathbb{Z}^n).$$
So, by the definition of $T_\sigma^t$ and the duality of 
$\ell^2(\Zn),$ $\varphi \in \mathcal{D}(T_\sigma^t)$ and $T_\sigma^t \varphi= T_\sigma^* \varphi.$ 
\end{proof}

\begin{prop} \label{prep11}
	$T_{\sigma, 1}$ is the largest closed extension of $T_\sigma$ in the 
sense that if $B$ is any extension of $T_\sigma$ such that 
$\mathcal{S}(\mathbb{Z}^n) \subseteq \mathcal{D}(B^t),$ then $T_{\sigma, 1}$ is an extension of $B.$ 
\end{prop}
\begin{proof}
	Let $u \in \mathcal{D}(B)$. Then for all $\psi \in 
\mathcal{S}(\mathbb{Z}^n),$ we have $\psi \in \mathcal{D}(B^t).$ Hence by 
the definition of $B^t,$ 
	\begin{equation}    \label{eq5}
	( \psi, Bu)_{\ell^2(\Zn)} =(B^t \psi, \varphi)_{\ell^2(\Zn)}.
	\end{equation}
	Since $B$ is an extension of $T_\sigma,$ it follows from 
\cite[Proposition 13.9]{Wongbook} that $T_\sigma^t$ is an extension of $B^t.$ 
Hence by \eqref{eq5},
	\begin{equation} \label{eq6}
	(\psi, Bu)_{\ell^2(\Zn)}  = ( T_\sigma^t \psi, u)_{\ell^2(\Zn)}.
	\end{equation}
	By Lemma \ref{lem10}, $T_\sigma^t= T_\sigma^*$ on 
$\mathcal{S}(\mathbb{Z}^n).$ Hence  by (\ref{eq6}), we have 
	$$ ( \psi, Bu )_{\ell^2(\Zn)}= (T_\sigma^* \psi, 
u)_{\ell^2(\Zn)},\quad \psi 
\in \mathcal{S}(\mathbb{Z}^n).$$ Therefore by Definition \ref{def4.6}, we have $u \in \mathcal{D}(T_{\sigma, 1})$ and $T_{\sigma, 1}u=Bu.$
\end{proof}

In view of of Proposition \ref{prep11}, we call $T_{\sigma,1}$  the {\it 
maximal operator} of $T_\sigma.$

\begin{prop} \label{Dense}
	Let $\sigma \in S^m(\Z^n\times \mathbb{T}^n).$ Then  the 
pseudo-differential operator  
$T_\sigma$ maps $\mathcal{S}(\mathbb{Z}^n)$ into $\mathcal{S}(\mathbb{Z}^n).$
\end{prop}
\begin{proof}
	We need to show that for all $\varphi$ in $\S(\Z^n)$, $$\sup_{k\in 
\mathbb{Z}^n } |k^\alpha (\Delta_{k}^\beta T_\sigma \varphi)(k)|<\infty.$$
	Using the Leibniz formula, summation by parts and integration by 
parts, we get
	\begin{align*}
k^\alpha  (\Delta_{k}^\beta T_\sigma\varphi)(k) &= k^\alpha 
\int_{\mathbb{T}^n} {\Delta}_{k}^\beta(e^{2 \pi i k \cdot x} \sigma 
(k, x))\, \widehat{\varphi}(x) \,dx \\ &= 	k^\alpha 
\int_{\mathbb{T}^n} \sum_{\gamma \leq \beta} \binom{\beta}{\gamma} 
(\Delta_{k}^\gamma e^{2 \pi i k \cdot x}) \bar{\Delta}_{k}^{\beta-\gamma} 
\sigma(k+\beta, x) \widehat{\varphi}(x)\, dx \\&= k^\alpha 
\int_{\mathbb{T}^n}  \sum_{\gamma \leq \beta} (2 \pi i)^{|\gamma|} x^ 
\gamma e^{2 \pi i k \cdot x} \bar{\Delta}_{k}^{\beta-\gamma} 
\sigma(k+\beta, x) \widehat{\varphi}(x)\, dx \\& \leq  (2 \pi 
i)^{-|\alpha|} \int_{\mathbb{T}^n} \sum_{\gamma \leq \beta} (2 \pi 
i)^{|\gamma|} \binom{\beta}{\gamma} D_x^\alpha(e^{2 \pi i k \cdot x}) 
\bar{\Delta}_k^{\beta -\gamma} \sigma(k+\beta, x) x^\gamma 
\widehat{\varphi}(x)\, dx \\&=   (2 \pi i)^{-|\alpha|}  (-1)^{|\alpha|} 
\int_{\mathbb{T}^n} \sum_{\gamma \leq \beta} (2 \pi i)^{|\gamma|} 
\binom{\beta}{\gamma} e^{2 \pi i k \cdot x} D_x^\alpha 
(\bar{\Delta}_k^{\beta -\gamma} \sigma(k+\beta, x) x^\gamma \widehat{\varphi}(x))\, dx
\\&=  (2 \pi i)^{-|\alpha|}  (-1)^{|\alpha|} \int_{\mathbb{T}^n} 
\sum_{\gamma \leq \beta} (2 \pi i)^{|\gamma|} \sum_{\delta \leq \alpha} 
\binom{\beta}{\gamma} \binom{\alpha}{\delta} 
e^{2 \pi i k \cdot x}  \\ &\times D_x^{\alpha-\delta}
( (\bar{\Delta}_k^{\beta -\gamma} \sigma)(k+\beta, x) 
D_x^\delta (x^\gamma \widehat{\varphi}(x))\, dx.
	\end{align*} 
	Using the fact that $\sigma \in S^m(\Z^n\times 
\mathbb{T}^n),$ there exists a positive constant 
$C_{\alpha, \beta, \gamma, \delta}>0$ such that 
	$$ \sup_{k\in \mathbb{Z}^n } |k^\alpha (\Delta_{k}^\beta  T_\sigma 
\varphi)(k)| \leq (2 \pi)^{|\beta|} \sum_{\gamma \leq \beta} \sum_{\delta 
\leq \alpha} \binom{\beta}{\gamma} \binom{\alpha}{\delta} C_{\alpha, \beta, 
\gamma, \delta} \int_{\mathbb{T}^n} (|k|+1|)^{m-|\alpha+|\delta|}  
D_x^\delta (x^\gamma \widehat{\varphi}(x))\, dx. $$
	 Since $\varphi \in \mathcal{S}(\mathbb{Z}^n)$,  we have 
	$$ \sup_{k\in \mathbb{Z}^n } |k^\alpha (\Delta_{k}^\beta  T_\sigma 
\varphi)(k)|<\infty.$$
Hence $T_\sigma \varphi \in \mathcal{S}(\mathbb{Z}^n).$
\end{proof}

\begin{lem} \label{swar}
	The space $\mathcal{S}(\mathbb{Z}^n)$ is dense in $H^{s, 2}(\mathbb{Z}^n).$
\end{lem}
\begin{proof}
	Let $u \in H^{s, 2}(\mathbb{Z}^n).$ By the definition of 
$H^{s, 2}(\mathbb{Z}^n),$ $J_{-s} u \in \ell^2(\Zn).$ Since 
$\mathcal{S}(\mathbb{Z}^n)$ is dense in $\ell^2(\Zn),$ it follows 
that there exists a sequence $\{\varphi_k\}$ in 
$\mathcal{S}(\mathbb{Z}^n)$ 
such that $\varphi_k \rightarrow J_{-s}u$ in $\ell^2(\Zn)$ as $k 
\rightarrow \infty.$ Let $\psi_k= J_s \varphi_k,\, k=1,2,\ldots.$ By 
Proposition  \ref{Dense}, $\psi_k \in \mathcal{S}(\mathbb{Z}^n)$ 
because $J_s$ is a pseudo-differential operator. By the definition of 
$H^{s, 2}(\mathbb{Z}^n)$, we have 
	$$\|\psi_k- u\|_{s,2} = \|J_{-s} \psi_k- J_{-s}u\|_2 = 
\|\varphi_k-J_{-s}u\|_2 \rightarrow 0$$ as $k \rightarrow \infty.$ This 
shows 
that $\psi_k \rightarrow u$ in $H^{s,2}(\Z^n)$ as $k\to \infty,$ and hence 
$\mathcal{S}(\mathbb{Z}^n)$ is dense in $H^{s, 2}(\mathbb{Z}^n).$ 
\end{proof}

The following theorem is a discrete version of Agmon--Douglis--Nirenberg 
estimate for pseudo-differential operators on $\mathbb{Z}^n.$
\begin{prop} \label{vis14}
	Let $m>0$ and let $\sigma$ be  an elliptic symbol in 
$S^m(\Z^n\times \mathbb{T}^n).$ 
Then there exist positive constants $C_1$ and $C_2$ such that 
	$$ C_1 \|u\|_{m, 2} \leq ( \|T_\sigma u\|_2+\|u\|_2) \leq 
C_2 \|u\|_{m,2},\quad u \in H^{m,2}(\Z^n).$$
\end{prop}
\begin{proof}
	By Theorem \ref{vis4.3} and Theorem \ref{vis4.4}, there exists a positive constant $C_2$ such that 
	\begin{equation}
	(\|T_\sigma u\|_{2}+ \|u\|_2) \leq C_2 \|u\|_{m, 2},\quad  u \in 
H^{m,2}(\mathbb{Z}^n).
	\end{equation}
	Now, by Theorem \ref{Ruz3.6} we have 
	   \begin{equation}
	   u= T_\tau T_\sigma u - Ru,\quad  u \in H^{m,2}(\mathbb{Z}^n),
	   \end{equation} where $\tau \in S^{-m}(\Z^n\times \mathbb{T}^n)$ 
and $R$ is 
a pseudo-differential operator with symbol in $\cap_{\nu \in \R} 
S^\nu(\Z^n\times \mathbb{T}^n).$ 
Hence it follows from Theorem \ref{vis4.3} that there exists a positive constant $C_1$ such that 
	   $$C_1 \|u\|_{m,2 } \leq (\|T_\sigma u\|_2+ \|u\|_2).$$ This complete the proof.
\end{proof}

\begin{prop} \label{vis15}
	Let $m>0$ and let $\sigma$ be an elliptic symbol in 
$S^m(\Z^n\times \mathbb{T}^n).$ 
Then $\mathcal{D}(T_{\sigma,0})= H^{m,2}(\mathbb{Z}^n).$  
\end{prop}
\begin{proof} 
	Let $u \in H^{m,2}(\mathbb{Z}^n).$ Then by Lemma \ref{swar}, we can 
find a sequence  $\{\varphi_k\}$ in $\mathcal{S}(\mathbb{Z}^n)$ such that 
$\varphi_k \rightarrow u$ in $H^{m, 2}(\mathbb{Z}^n)$ as $k \rightarrow 
\infty.$ By Lemma \ref{swar} and Proposition \ref{vis14},  $\{T_\sigma 
\varphi_k\}$ and $\{\varphi_k\}$ are Cauchy sequences in 
$\ell^2(\Zn)$ and hence $\varphi_k \rightarrow u$ and $T_\sigma 
\varphi_k 
\rightarrow f$ in $\ell^2(\mathbb{Z}^n)$ for some $u$ and $f$ in 
$\ell^2(\mathbb{Z}^n)$ as $k\to \infty$.  Hence by the definition of 
$T_{\sigma, 
0},$ $u \in \mathcal{D}(T_{\sigma, 0})$ and $T_{\sigma, 0}u =f.$
On the other hand, let  $u \in \mathcal{D}(T_{\sigma, 0}).$ Then by 
the definition of $T_{\sigma, 0}$ again, we can find a 
sequence $\{\varphi_k\}$ in 
$\mathcal{S}(\mathbb{Z}^n)$ such that $\varphi_k \rightarrow u$ and 
$T_\sigma \varphi_k \rightarrow f$ in $\ell^2(\Zn)$ for some $f 
\in 
\ell^2(\Zn)$ as $k\to \infty.$ Therefore $\{\varphi_k\}$ and 
$\{T_\sigma \varphi_k\}$ are 
Cauchy sequences in $L^2(\mathbb{Z}^2).$ So, by Proposition \ref{vis14} 
and Lemma \ref{swar}, $\{\varphi_k\}$ is a Cauchy sequence in $H^{m, 
2}(\mathbb{Z}^n).$ Since $H^{m, 2}(\mathbb{Z}^n)$ is complete, it follows 
that $\varphi_k \rightarrow v$ in $H^{m,2}(\mathbb{Z}^n)$ for some $v \in 
H^{m,2}(\mathbb{Z}^n)$ as $k \rightarrow \infty.$ Then 
by Theorem \ref{vis4.4}, $\phi_k \rightarrow v$ in $\ell^2(\Zn)$ as 
$k \rightarrow \infty.$ Hence $u=v$ and consequently, $u 
\in H^{m,2}(\mathbb{Z}^n).$
\end{proof}

The following theorem is the main theorem of this section.

\begin{thm}
	Let $m>0$ and let $\sigma$ be an elliptic symbol in 
$S^m(\Z^n\times \mathbb{T}^n).$ Then $T_{\sigma, 0}=T_{\sigma, 1}.$
\end{thm} 
\begin{proof}
Since $T_{\sigma, 0}$ is the smallest closed extension	of $T_\sigma,$ 
it follows from Proposition \ref{vis15} that it is sufficient to prove that 
$\mathcal{D}(T_{\sigma, 0}) \subseteq H^{m,2}(\mathbb{Z}^n).$ Let $u \in 
\mathcal{D}(T_{\sigma, 1}).$ Then Theorem \ref{Ruz3.6} gives   
\begin{equation} \label{lastvis}
u= T_\sigma T_\tau u -Ru,
\end{equation} where $\tau \in S^{-m}(\Z^n\times \mathbb{T}^n)$ and $R$ is 
a 
pseudo-differential 
operator with symbol in $\cap_{k \in \R} S^k(\Z^n\times \mathbb{T}^n.$ By 
Proposition 
\ref{vis4.7}, $T_{\sigma, 1}u= T_\sigma u$ in the distribution sense. 
Thus, by the  definition of $T_{\sigma, 1}$, $ u \in 
\ell^2(\Zn).$ Since 
$\tau \in S^{-m}(\Z^n\times \mathbb{T}^n),$ it follows  from Theorem 
\ref{compo} and Theorem 
\ref{vis4.3} that $T_\sigma T_\tau u \in H^{m, 2}(\mathbb{Z}^n).$ Since 
$u \in \ell^2(\mathbb{Z}^n)$ and $R$ is a pseudo-differential operator with 
symbol in $S^{-m}(\Z^n\times \mathbb{T}^n)$, it follows from Theorem 
\ref{vis4.3} again that 
$Ru \in H^{m, 2}(\mathbb{Z}^n).$ Hence $u \in H^{m, 2}(\mathbb{Z}^n).$
	
\end{proof}

\section{ Fredholmness and ellipticity}\label{Fredell}
In this section we study the Fredholmness and ellipticity of a pseudo-differential operator. We prove that the a pseudo-differential operator of order 0 is elliptic if and only if Fredholm. We also calculate the index of such a pseudo-differential operator.

Let us first recall that a closed linear operator $A$ from a complex 
Banach space $X$ into a complex Banach space $Y$ with dense domain $\mathcal{D}(A)$ is said to be Fredholm if the range $R(A)$ of $A$ is a closed subspace of $Y$, the null space $N(A)$ of $A$ and the null space $N(A^t)$ of the true adjoint $A^t$ of $A$ are finite dimensional. For a Fredholm operator $A$, the index $i(A)$ of $A$ is defined by
$$i(A)=\dim N(A)-\dim N(A^t).$$

The following criterion for a closed linear operator to be Fredholm is 
usually attributed to Atkinson \cite{At}.
\begin{thm}
Let $A$ be a closed linear operator from a complex Banach space $X$ into a 
complex Banach space $Y$ with dense domain $\mathcal{D}(A).$ Then $A$ is 
Fredholm if and only if we can find a bounded linear operator 
$B:Y\rightarrow X,$ a compact operator $K_{1}:X\rightarrow X$ and a 
compact operator $K_2:Y\rightarrow Y$ such that $BA=I+K_1$ on $\mathcal{D}(A)$ and $AB=I+K_2$ on $Y$.
\end{thm}

The main result  in this section is the following theorem.
\begin{thm}\label{Freeqelli}
	Let $\sigma \in S^0(\mathbb{Z}^n \times \mathbb{T}^n)$. 
Then $T_\sigma: \ell^2(\Zn) \rightarrow \ell^2(\Zn)$ is 
Fredholm if and only if $T_\sigma:\ell^2(\Zn)\to \ell^2(Z^n)$ is elliptic. 
\end{thm}
\begin{proof}
	Suppose that $T_\sigma:\ell^2(\Zn)\to \ell^2(\Zn)$ is a Fredholm 
operator.  Then $T_\tau= 
\mathcal{F}_{\mathbb{Z}^n}  T_\sigma^t  \mathcal{F}^{-1}_{\Z^n}$ is also a 
Fredholm operator from $L^2(\mathbb{T}^n)$ into $L^2(\mathbb{T}^n)$ , 
where $$\tau(x, k)= \overline{\sigma(-k, x)},\quad x\in \mathbb{T}^n, 
k\in \Z^n.$$ It 
is 
easy to see that $\sigma \in S^0(\mathbb{Z}^n \times \mathbb{T}^n)$ 
implies that $\tau \in S^0(\mathbb{T}^n \times \mathbb{Z}^n).$ Therefore 
by \cite[Theorem 3.9]{ Mol10},  $\tau$ is elliptic which is 
equivalent to say that $\sigma$ is elliptic. Conversely, suppose that 
$\sigma$ is 
elliptic. Then therse exists a symbol $\tau\in S^{0}(\Z^n\times 
\mathbb{T}^n)$ such 
that
\begin{equation}
T_{\sigma}T_{\tau}=I+R
\end{equation}
and
\begin{equation}
T_\tau T_\sigma=I+S,
\end{equation}
where $R$ and $S$ are pseudo-differential operators with symbol in 
$\cap_{\nu\in 
\R}S^{\nu}(\Z^n\times \mathbb{T}^n).$ So, for any positive number $k$, the 
pseudo-differential  operator $R:  
\ell^{2}(\mathbb{Z}^n)\rightarrow \ell^{2}(\mathbb{Z}^n)$ is the same as the 
composition of the psudo-differential operator $R  : 
\ell^{2}(\mathbb{Z}^n)\rightarrow 
H^{t,2}(\mathbb{Z}^n)$ and the 
inclusion $i:H^{t,2}(\mathbb{Z}^n)\rightarrow  
\ell^{2}(\mathbb{Z}^n). $ Since $R  : \ell^{2}(\mathbb{Z}^n)\rightarrow 
H^{t,2}(\mathbb{Z}^n)$ is a bounded linear operator and from 
Theorem \ref{compact} 
$i:H^{t,2}(\mathbb{Z}^n)\rightarrow  \ell^{2}(\mathbb{Z}^n) $ is compact,  it 
follows that $R  : \ell^{2}(\mathbb{Z}^n)\rightarrow  \ell^{2}(\mathbb{Z}^n)$ is 
compact. Similarly $S  : \ell^{2}(\mathbb{Z}^n)\rightarrow  
\ell^{2}(\mathbb{Z}^n)$ is compact.  Hence by Atkinson's theorem,  
$T_{\sigma}$ is Fredholm.
\end{proof}
We end this section with an index formula for  Fredholm 
pseudo-differential operators on $\Z^n$.
First,  we prove the following lemma.
\begin{lemma}\label{InSc}
Let $\mathcal{S}(\Z^n\times \mathbb{T}^n)$ be the Schwartz space 
on $\Z^n\times \mathbb{T}^n$. Then
$$\bigcap\limits_{\nu\in \mathbb{R} 
}S^{\nu}(\mathbb{Z}^{n}\times\mathbb{T}^n)
=\mathcal{S}(\Z^n\times \mathbb{T}^n ).$$  
\end{lemma}
\begin{proof} It is easy to check that 
$\mathcal{S}(\Z^n\times \mathbb{T}^n)\subseteq 
\bigcap\limits_{\nu\in \mathbb{R} }S^{\nu}(\Z^n\times \mathbb{T}^n)$. 
Now,  let $\sigma\in \bigcap\limits_{\nu\in \mathbb{R} 
}S^{\nu}(\mathbb{Z}^{n}\times\mathbb{T}^n). $ 
Let $\alpha$, $\beta$, $\gamma$ and $\delta$ be multi-indices. Then
there exists a real 
number  $\nu_1$ such that
$$|\alpha|+\nu_1-|\gamma|\leq 0.$$
Since $\sigma\in S^{\nu_1},$ we can find a positive constant $C_{\nu_1, 
\gamma, \delta}$ such that
\begin{eqnarray}\sup_{(k,x)\in 
\Z^n\times 
\mathbb{T}^n}|k^{\alpha}x^{\beta}(D_{x}^{\delta}\Delta_{k}^{\gamma}\sigma)
(k,x)|\nonumber\\
&\leq& \sup_{(k,x)\in \Z^n\times \mathbb{T}^n} 
|k|^{|\alpha|}|(D_{x}^{\delta}\Delta_{k}^{\gamma}\sigma)(k,x)|\nonumber\\
&\leq& C_{\nu_1,\gamma,\delta}\sup_{k\in \mathbb{Z}^n, x\in \mathbb{T}^n 
}(1+|k|)^{\nu+|\alpha|-|\gamma|}<\infty
\end{eqnarray}
Therefor $\sigma\in \S(\Z^n\times \mathbb{T}^n).$
\end{proof}

Let $\sigma \in S^0(\mathbb{Z}^n \times \mathbb{T}^n)$ be elliptic. Let 
$\tau\in S^{0}(\Z^n\times \mathbb{T}^n)$ be such that
$$T_{\tau}T_{\sigma}=I-T_{1}$$ and $$T_{\sigma}T_{\tau}=I-T_{2},$$ where 
$T_1$ and $T_2$ are pseudo-differential operators with symbols $
\tau_1$ and
$\tau_2$ respectively in $\bigcap\limits_{\nu\in \mathbb{R} 
}S^{\nu}(\mathbb{Z}^{n}\times\mathbb{T}^n).$ Let $
f \in\mathcal{S}$. Then for $j=1,2,$
\begin{eqnarray}
(T_{j}f)(k)&=&\int_{\mathbb{T}^n}e^{2\pi ik\cdot x}\tau_j(k,x)\widehat{f}(x)
dx\nonumber
\end{eqnarray}
for all $k\in \mathbb{Z}^n.$  By Lemma \ref{InSc}, for $j=1,2$, $\tau_j\in 
\mathcal{S}(\Z^n\times \mathbb{T}^n)$.  Then by Theorem 5.3 in  
\cite{Rulat}, $T_1$ and $T_2$  are trace class operators and 
$$\text{tr}(T_j)=\sum\limits_{k\in\mathbb{Z}^n}\int_{\mathbb{T}^n}\tau_j(k,x)
dx.$$  
Then by Theorem 20.13 in \cite{Wongbook}, we get
$$i(T_{\sigma})=\text{tr}(T_1)-\text{tr}({T_2})=\sum\limits_{k\in\mathbb{Z}^n}\int_{\mathbb{T}^n}(\tau_1-\tau_2)(k,x)
dx.$$

\section{Acknowledgement}
We are grateful to Professor M. W. Wong for carefully reading the paper and the very useful suggestions.
Vishvesh Kumar is supported by Odysseus I Project (by FWO, Belgium) of 
Prof. Michael Ruzhansky. He thanks Prof. Michael Ruzhansky for his support 
and encouragement.


\begin{thebibliography}{aaa}
	
	\normalsize
	\baselineskip=17pt
	
\bibitem{At} F. V. Atkinson, The normal solubility of linear equations in normed spaces (Russian), Mat. Sbornik N. S. 28(70) , 3-14, 1951.
\bibitem{Rulat} L. N. A. Botchway, P. G. Kabiti and M. Ruzhansky, Difference equations and pseudo-differential operators on $\mathbb{Z}^n,$ 2017 (preprint) arXiv:1705.07564 
\bibitem{CK} D. Cardona and V. Kumar, $L^p$-boundedness and $L^p$-nuclearity of multilinear pseudo-differential operators on $\Zn$ and $\mathbb{T}^n,$ \emph{J. Fourier Anal. Appl.}, 2019. https://doi.org/10.1007/s00041-019-09689-7
\bibitem{Duvan} D. Cardona, Pseudo-Differential Operators on $Zn$  with Applications to Discrete Fractional Integral Operators, \emph{Bull. Iran. Math. Soc.}, 45(4) 1227-1241, 2019.
\bibitem{Cor95} H. O. Cordes. The technique of pseudodifferential operators,  \emph{ London Mathematical Society} Lecture Note Series. Cambridge University Press, Cambridge,volume 202, 1995.
\bibitem{AW2} A. Dasgupta and M.W. Wong, Ellipticity of Fredholm Pseudo-Differential Operators on $L^p(\Rn)$, \emph{New Developments in Pseudo-Differential Operators. Operator Theory: Advances and Applications}, vol 189. Birkh\"{a}user Basel.
\bibitem{ADWo} A. Dasgupta and M.W.Wong, Spectral Theory of SG Pseudo-Differential Operators on $L^p(\Rn)$, \emph{Studia Mathematica}, 187 (2),  185-197, 2008. 
\bibitem{DW13} J. Delgado and M. W. Wong. $L^p$-nuclear pseudo-differential operators on $\mathbb{Z}$ and $\mathbb{S}^{1}$ \emph{Proc. Amer. Math. Soc.,} 141(11): 3935–3942, 2013.

\bibitem{GJBNM16} M. B. Ghaemi, M. Jamalpour Birgani, and E. Nabizadeh Morsalfard. A study on
pseudo-differential operators on $\mathbb{S}^1$ and $\mathbb{Z}$. \emph{J. Pseudo-Differ. Oper. Appl.,} 7(2):237–
247, 2016.
\bibitem{Gru}  V. V. Grushin, Pseudo-differential operators on $\mathbb{R}^{n}$ with bounded symbols, Funct. Anal. 4, 202-212, 1970. 
\bibitem{Mol10} S. Molahajloo and M. W. Wong, Ellipticity, Fredholmness and spectral invariance of Pseudo-differential operators on $\mathbb{S}^1,$ \emph{J. Pseudo-Differ. Oper. Appl.} 1  183-205, 2010. 
\bibitem{SW} S. Molahajloo. Pseudo-differential operators on $\mathbb{Z}$. In Pseudo-differential operators: complex analysis and partial differential equations, volume 205 of Oper. Theory Adv.Appl., pages 213–221. Birkh\"{a}user Verlag, Basel, 2010.
\bibitem{Rab10} V. Rabinovich. Exponential estimates of solutions of pseudodifferential equations on the lattice $(h\mathbb{Z})^n$
applications to the lattice Schr\"{o}dinger and Dirac operators. \emph{J. Pseudo-Differ. Oper. Appl.,} 1(2):233–253, 2010.

\bibitem{RR09} V. S. Rabinovich and S. Roch. Essential spectra and exponential estimates of eigenfunctions of lattice operators of quantum mechanics. \emph{J. Phys. A,} 42(38):385207, 21, 2009.
\bibitem{RT11}   C. A. Rodriguez Torijano. $L^p$-estimates for pseudo-differential operators on $\mathbb{Z}^n$. \emph{J. Pseudo-Differ. Oper. Appl.}, 2(3):367–375, 2011.
\bibitem{R2} M. Ruzhansky and V. Turunen, On the Fourier analysis of operators on the torus, Modern trends in pseudo-differential operators, 87-105, Oper. Theory Adv. Appl., 172, Birkhauser, Basel, 2007. 
\bibitem{RT} M. Ruzhansky and V. Turunen, {\it Pseudo-differential oeprators and symmetries}, Birkh$\ddot{a}$user Verlag (2009).
\bibitem{R7} M. Ruzhansky and V. Turunen,, On the toroidal quantization of periodic pseudo-differential operators, Numerical Functional Analysis and Optimization, 30 (2009), 1098-1124.
\bibitem{R6} M. Ruzhansky, V. Turunen, Quantization of pseudo-differential operators on the torus, J. Fourier Anal. Appl., 16 (2010), 943-982.
\bibitem{R1}M. Ruzhansky,  V. Turunen, J. Wirth, Hörmander class of pseudo-differential operators on compact Lie groups and global hypoellipticity, J. Fourier Anal. Appl., 20 (2014), 476-499.
\bibitem{R3} M. Ruzhansky and V. Turunen, Sharp Garding inequality on compact Lie groups, J. Funct. Anal., 260 (2011), 2881-2901.
\bibitem{R4} M. Ruzhansky and J. Wirth, Global functional calculus for operators on compact Lie groups, J. Funct. Anal., 267 (2014), 144-172.


\bibitem{Wo17} M. W. Wong, Fredholm pseudo-differential operators on weighted Sobolev spaces, Ark. Math 21 , 271-282, 1983.
\bibitem{Wo18} M.W.Wong, Spectral theory of pseudo-differential operators. \emph{ Adv. in Appl. Math.} 15 no. 4, 437–451, 1994. 
\bibitem{Wo} M. W. Wong, M-elliptic pseudo-differential operators on $L^{p}(\mathbb{R}^n).$ \emph{Math. Nachr.},  279: 319-326 (2006). 
\bibitem{Wongbook} M. W. Wong,  \emph{An introduction to Pseudo-differential Operators}, World Scientific. 3rd Edition, 2014.  	
	
	
\end{thebibliography}
\end{document}